\documentclass{amsart}
\usepackage{amssymb, amsmath, amsthm, graphics, comment, xspace, enumerate, xcolor}
\newtheorem{thm}{Theorem}[section]

\newtheorem{prop}[thm]{Proposition}
\theoremstyle{definition}
\newtheorem{defn}[thm]{Definition}

\theoremstyle{remark}
\newtheorem{rem}[thm]{Remark}
\numberwithin{equation}{section}

\begin{document}

\title[A note on (asymptotically) Weyl-almost periodic...]{A note on (asymptotically) Weyl-almost periodic 
properties of convolution products}

\author{Vladimir E. Fedorov}
\address{Chelyabinsk State State University, Chelyabinsk, 454 080, Russia}
\email{kar@csu.ru}

\author{Marko Kosti\' c}
\address{Faculty of Technical Sciences,
University of Novi Sad,
Trg D. Obradovi\' ca 6, 21125 Novi Sad, Serbia}
\email{marco.s@verat.net}

\begin{abstract}
The main aim of this paper is to investigate Weyl-$p$-almost periodic 
properties and asymptotically Weyl-$p$-almost periodic 
properties of convolution products. In such a way, we continue several recent research studies of ours which do concern a similar problematic.
\end{abstract}



\maketitle

\section{Introduction and preliminaries}

In a series of recent research papers, the second named author has considered the invariance of (asymptotical) Weyl-$p$-almost periodicity 
 under the action of (finite) infinite 
convolution product, where $1\leq p<\infty.$ It has been perceived that the case $p>1$ is much more delicate for the analysis and, before proceeding any further, we would like to stress that 
Proposition 2.1 in \cite{irkutsk} is not correctly proved in the case that $p>1.$ In a recent erratum and addendum to the paper \cite{irkutsk}, we have introduced the class of  
quasi asymptotically almost periodic functions and considered quasi asymptotical almost periodicity of infinite convolution product 
\begin{align}\label{profa}
G(t)\equiv  t\mapsto \int_{-\infty}^{t}R(t-s)g(s)\, ds,\ t\in {\mathbb R},
\end{align}
where $p\geq 1$ and $g(\cdot)$ is (equi-)Weyl-$p$-almost periodic.

In this paper, we consider the invariance of (asymptotical) Weyl-$p$-almost periodicity 
under the action of (finite) infinite 
convolution product by assuming that the corresponding resolvent operator family $(R(t))_{t>0}\subseteq L(X,Y)$ has a certain growth order at zero and infinity (by $(X,\|\cdot\|),$ $(Y,\|\cdot\|_{Y})$ and $L(X,Y)$ we denote two non-trivial complex Banach spaces and the space consisting of all linear continuous operators from $X$ into $Y,$ respectively; in the sequel, we will use the standard terminology from the monograph \cite{knjigaho}). 
We specifically consider the following two types of growth rates:
\begin{align}\label{isti-e}
\|R(t)\|_{L(X,Y)}\leq Me^{-ct}t^{\beta -1},\ t>0\mbox{ for some finite constants }c>0,\ \beta \in (0,1],\ M>0,
\end{align}
or a substantially weaker one
\begin{align}\label{isti-ee}
\|R(t)\|_{L(X,Y)}\leq M\frac{t^{\beta -1}}{1+t^{\gamma}},\ t>0\mbox{ for some finite constants }\gamma>1,\ \beta \in (0,1],\ M>0.
\end{align}
The estimate \eqref{isti-e} appears in the theoretical studies of abstract degenerate differential equations of first order with multivalued linear operators ${\mathcal A}$ satisfying the condition (P) clarifed below, while the estimate
\eqref{isti-ee} appears in the theoretical studies of abstract degenerate fractional relaxation differential equations with multivalued linear operators ${\mathcal A}$ satisfying the same condition. 

The genesis of paper is stimulated by reading some recent results of 
F. Bedouhene, N. Challali, O. Mellah, P. Raynaud de Fitte and M. Smaali, which are still unpublished and where the cases that $X=Y$ and $(R(t))_{t\geq 0}$ is an exponentially decaying strongly continuous non-degenerate $C_{0}$-semigroup have been analyzed. The main novelty of this paper is the consideration of growth 
rate \eqref{isti-ee} for solution operator families $(R(t))_{t>0}$ not necessarily strongly continuous at zero, with regards to the existence and uniqueness of (asymptotically) Weyl-$p$-almost periodic solutions of
abstract fractional differential equations; see e.g. \cite{faviniyagi}-\cite{fedorov-primonja}, \cite{afid-mlo} and \cite{fedorov} for some recent results treating the abstract degenerate fractional differential equations.

There is no need to say that fractional calculus and fractional differential equations are rapidly growing fields of research, due to their invaluable importance in modeling real world phenomena appearing in many fields such
as astrophysics, electronics, diffusion, chemistry, biology,
electricity and thermodynamics. It is almost impossible to summarize here all relevant contributions made recently in these fields; see e.g. \cite{kilbas}-\cite{knjigaho}, \cite{samko} and references cited therein for the basic and not-updated information on the subject. Throughout the paper, we use two different types of fractional derivatives.
The Weyl-Liouville fractional derivative $D_{t,+}^{\gamma}u(t)$ of order $\gamma  \in (0,1)$ is defined for those continuous functions
$u : {\mathbb R} \rightarrow X$
satisfying that $t\mapsto \int_{-\infty}^{t}g_{1-\gamma}(t-s)u(s)\, ds,$ $t\in {\mathbb R}$ is a well-defined continuously differentiable mapping, by
$$
D_{t,+}^{\gamma}u(t):=\frac{d}{dt}\int_{-\infty}^{t}g_{1-\gamma}(t-s)u(s)\, ds,\quad t\in {\mathbb R}.
$$
Set $
D_{t,+}^{1}u(t):=-(d/dt)u(t).$
For further information about Weyl-Liouville fractional derivatives, we refer the reader to the paper \cite{relaxation-peng} by
J. Mu, Y. Zhoa and L. Peng. 

If $\alpha >0$ and $m=\lceil \alpha \rceil,$ then 
the Caputo fractional derivative\index{fractional derivatives!Caputo}
${\mathbf D}_{t}^{\alpha}u(t)$ is defined for those functions $u\in
C^{m-1}([0,\infty) : X)$ satisfying that $g_{m-\alpha} \ast
(u-\sum_{k=0}^{m-1}u_{k}g_{k+1}) \in C^{m}([0,\infty) : X),$
by
$$
{\mathbf
D}_{t}^{\alpha}u(t)=\frac{d^{m}}{dt^{m}}\Biggl[g_{m-\alpha}
\ast \Biggl(u-\sum_{k=0}^{m-1}u_{k}g_{k+1}\Biggl)\Biggr].
$$
For more details about the abstract fractional differential equations with Caputo derivatives, the reader may consult the monograph \cite{knjigaho} and references cited therein.

The organization of this paper can be simply described as follows. In Subsection \ref{prcko}, we recall the basic definitions and results about generalized (asymptotically) almost periodic functions which will be necessary for our further work. Our main contributions are Theorem \ref{brzo-kuso} and Proposition \ref{brzo-kusowas}; besides them, Section \ref{prcko-prim} contains a great deal of other remarks and observations
about problems considered. It is clear that our results are applicable in the analysis of a wide class of abstract inhomogenous integro-differential equations, which can be degenerate or non-degenerate in time-variable and which may or may not contain fractional derivatives. The main aim of Section \ref{prcko-duo} is to present certain applications of Theorem \ref{brzo-kuso} and Proposition \ref{brzo-kusowas} in the analysis of existence and uniqueness of Weyl-$p$-almost periodic solutions of the abstract fractional differential inclusion \eqref{inkpow} and asymptotically Weyl-$p$-almost periodic solutions of the abstract fractional differential inclusion (DFP)$_{f,\gamma}$ clarifed below. 

\subsection{Generalized almost periodic functions and asymptotically generalized almost periodic functions}\label{prcko}

Unless stated otherwise, we will always assume henceforth that $1\leq p<\infty .$ Let $I=[0,\infty)$ or $I={\mathbb R}.$
A function $f\in L^{p}_{loc}(I :X)$ is said to be Stepanov $p$-bounded iff
$$
\|f\|_{S^{p}}:=\sup_{t\in I}\Biggl( \int^{t+1}_{t}\|f(s)\|^{p}\, ds\Biggr)^{1/p}<\infty.
$$
The space $L_{S}^{p}(I:X)$ consisted of all $S^{p}$-bounded functions becomes a Banach space equipped with the above norm. 

Let $1\leq p <\infty,$ let $l>0,$ and let $f,\ g\in L^{p}_{loc}(I :X).$ We define the Stepanov `metric' by 
\begin{align*}
D_{S_{l}}^{p}\bigl[f(\cdot),g(\cdot)\bigr]:= \sup_{x\in I}\Biggl[ \frac{1}{l}\int_{x}^{x+l}\bigl \| f(t) -g(t)\bigr\|^{p}\, dt\Biggr]^{1/p}.
\end{align*} 
Then it is well-known that there exists
\begin{align*}
D_{W}^{p}\bigl[f(\cdot),g(\cdot)\bigr]:=\lim_{l\rightarrow \infty}D_{S_{l}}^{p}\bigl[f(\cdot),g(\cdot)\bigr]
\end{align*}
in $[0,\infty].$
The distance $D_{W}^{p}[f(\cdot),g(\cdot)]$ appearing above is called the Weyl distance of $f(\cdot)$ and $g(\cdot).$ The Stepanov and Weyl `norm' of $f(\cdot)$ are defined by
$$
\bigl\| f  \bigr\|_{S_{l}^{p}}:= D_{S_{l}}^{p}\bigl[f(\cdot),0\bigr]\ \mbox{  and  }\ \bigl\| f  \bigr\|_{W^{p}}:= D_{W}^{p}\bigl[f(\cdot),0\bigr],
$$ 
respectively. The notions of Stepanov $p$-boundedness and Weyl $p$-boundedness are mutually equivalent, i.e., for any function $f\in L^{p}_{loc}(I :X),$ we have
$$
\bigl\| f  \bigr\|_{S_{l}^{p}}<\infty \ \ \mbox{ iff } \ \ \ \bigl\| f  \bigr\|_{W^{p}}<\infty .
$$

The notion of an (equi-)Weyl-$p$-almost periodic function is given below (see e.g. \cite{irkutsk} and references cited therein).

\begin{defn}\label{weyl-defn}
Let $f\in L_{loc}^{p}(I: X).$ 
\begin{itemize}
\item[(i)] We say that the function $f(\cdot)$ is equi-Weyl-$p$-almost periodic, $f\in e-W_{ap}^{p}(I:X)$ for short, iff for each $\epsilon>0$ we can find two real numbers $l>0$ and $L>0$ such that any interval $I'\subseteq I$ of length $L$ contains a point $\tau \in  I'$ such that
\begin{align}
\notag
\sup_{x\in I}\Biggl[  \frac{1}{l}\int_{x}^{x+l}\bigl \| f(t+\tau)& -f(t)\bigr\|^{p}\, dt\Biggr]^{1/p} \leq \epsilon, \\ \notag & \mbox{ i.e., } D_{S_{l}}^{p}\bigl[f(\cdot+\tau),f(\cdot)\bigr] \leq \epsilon.
\end{align} 
\item[(ii)] We say that the function $f(\cdot)$ is Weyl-$p$-almost periodic, $f\in W_{ap}^{p}(I: X)$ for short, iff for each $\epsilon>0$ we can find a real number $L>0$ such that any interval $I'\subseteq I$ of length $L$ contains a point $\tau \in  I'$ such that
\begin{align*}
\notag
\lim_{l\rightarrow \infty} \sup_{x\in I}\Biggl[ \frac{1}{l}\int_{x}^{x+l}\bigl \| f(t+\tau) & -f(t)\bigr\|^{p}\, dt\Biggr]^{1/p} \leq \epsilon, \\  & \mbox{ i.e., } \lim_{l\rightarrow \infty}D_{S_{l}}^{p}\bigl[f(\cdot+\tau),f(\cdot)\bigr] \leq \epsilon.
\end{align*} 
\end{itemize}
\end{defn}

Denote by $APS^{p}(I:X)$ the space consisting of all Stepanov $p$-almost periodic functions from the interval $I$ into $X$ (\cite{EJDE}). Then it is well known that 
$
APS^{p}(I:X) \subseteq e-W_{ap}^{p}(I:X) \subseteq W_{ap}^{p}(I:X)
$
in the set theoretical sense and that any of these two inclusions can be strict.

Denote by $C_{0}([0,\infty):X)$ the vector space consisting of all bounded continuous functions from $[0,\infty)$ into $X$ which vanish at infinity. We say that an $S^{p}$-bounded function
$q : [0,\infty) \rightarrow X$ is Stepanov $p$-vanishing iff the function
$t\mapsto q(t+\cdot),$ $t\geq 0$ belongs to the class $C_{0}([0,\infty) : L^{p}([0,1]:X)).$ We denote by $S^{p}_{0}([0,\infty) :X)$ the vector space consisting of all Stepanov $p$-vanishing functions.
If $q\in L_{loc}^{p}([0,\infty ) : X),$ then we define the function ${\bf q}(\cdot,\cdot): [0,\infty) \times [0,\infty) \rightarrow X$ by
$$
{\bf q}(t,s):=q(t+s),\quad t,\, s\geq 0.
$$ 
The class of (equi-)Weyl-$p$-vanishing functions has been recently introduced as follows (see e.g. \cite{irkutsk}):

\begin{defn}\label{stea-weyl}
Let $q\in L_{loc}^{p}([0,\infty ) : X).$
\begin{itemize}
\item[(i)]
It is said that $q(\cdot)$ is Weyl-$p$-vanishing iff
\begin{align*}
\lim_{t\rightarrow \infty}\bigl\|{\bf q}(t,\cdot)\bigr\|_{W^{p}}=0,\mbox{ i.e., }\lim_{t\rightarrow \infty}\, \lim_{l\rightarrow \infty}\ \ \sup_{x\geq 0}\Biggl[ \frac{1}{l}\int_{x}^{x+l}\bigl \| q(t+s)\bigr\|^{p}\, ds\Biggr]^{1/p}=0.
\end{align*}
\item[(ii)] It is said that $q\in L_{loc}^{p}([0,\infty ) : X)$ is equi-Weyl-$p$-vanishing iff
\begin{align*}
\lim_{l\rightarrow \infty}\, \lim_{t\rightarrow \infty}\ \ \sup_{x\geq 0}\Biggl[ \frac{1}{l}\int_{x}^{x+l}\bigl \| q(t+s)\bigr\|^{p}\, ds\Biggr]^{1/p}=0.
\end{align*}
\end{itemize}
\end{defn}

We know that $C_{0}([0,\infty):X) \subseteq S^{p}_{0}([0,\infty):X) \subseteq e-W^{p}_{0}([0,\infty):X)\subseteq W^{p}_{0}([0,\infty):X)$ and that any of these three inclusions can be strict.  

\section[Formulation and proof of main results]{Formulation and proof of main results}\label{prcko-prim}

Albeit given with a relatively non-complicated proof, the following theorem can be viewed as the main result of this paper:

\begin{thm}\label{brzo-kuso}
Let $1/p+1/q=1$ and let $(R(t))_{t>0}\subseteq L(X,Y)$ satisfy \eqref{isti-ee}. Let a function $g : {\mathbb R} \rightarrow X$ be (equi-)Weyl-$p$-almost periodic and Weyl $p$-bounded, and let $q(\beta-1)>-1$ provided that $p>1$, resp. $\beta=1,$ provided that $p=1$. Then the function
$G : {\mathbb R} \rightarrow Y,$ defined through \eqref{profa}, is bounded continuous and (equi-)Weyl-$p$-almost periodic.
\end{thm}

\begin{proof}
We will consider the case that $g(\cdot)$ is Weyl-$p$-almost periodic with $p>1$ and explain the main differences in the case that $p=1.$ Without loss of generality, we may assume that $X=Y.$ Since we have assumed that $g(\cdot)$ is Weyl $p$-bounded (equivalently, Stepanov $p$-bounded) and $q(\beta-1)>-1$, we can repeat literally the arguments given in the proof of \cite[Proposition 2.1]{irkutsk} to deduce that $G(\cdot)$ is bounded and continuous on the real line (a similar argumentation works in the case that $p=1$). Therefore, it remains to be proved that $G(\cdot)$ is Weyl-$p$-almost periodic. 

In order to do that, fix a number $\epsilon>0.$ By definition, we can find a real number $L>0$ such that any interval $I'\subseteq I$ of length $L$ contains a point $\tau \in  I'$ such that there exists a number $l(\epsilon,\tau)>0$ so that
\begin{align}\label{austrija}
D_{S_{l}}^{p}\bigl[g(\cdot+\tau),g(\cdot)\bigr] \leq \epsilon,\  l\geq l(\epsilon,\tau).
\end{align}
On the other hand, it is clear that
\begin{align}
\notag \|u(t+\tau)- u(t)\|\leq \int_{-\infty}^{0}\|R(-s)\| \|g(s+t+\tau)-g(s+t)\|\, ds
\\\label{medijana}\leq M \int_{-\infty}^{0}|s|^{\beta-1} \|g(s+t+\tau)-g(s+t)\| / (1+|s|^{\gamma}) \, ds,\ t\in {\mathbb R}.
\end{align}
Since $\gamma>1$ and $\beta \in (0,1],$ we have the existence of a positive real number $\zeta >0$ 
satisfying
$$
\frac{1}{p}<\zeta <\frac{1}{p}+\gamma-\beta
$$
(in the case that $p=1,$ we can take any number $\zeta \in (1,\gamma)$ and repeat the same procedure).
This implies that the function $s\mapsto \frac{|s|^{\beta-1}(1+|s|^{\zeta})}{1+|s|^{\gamma}},$ $s<0$ belongs to the space $L^{q}((-\infty,0))$ as well as that the function $s\mapsto \frac{1}{1+|s|^{\zeta}},$ $s<0$ belongs to the space $L^{p}((-\infty,0)).$ 
The integral
$$
\int^{0}_{-\infty}\|g(s+t+\tau)-g(s+t)\|^{p} / (1+|s|^{\zeta})^{p} \, ds
$$
converges for any $t\in {\mathbb R}$, which follows from the following computation
\begin{align*}
& \int^{0}_{-\infty}\|g(s+t+\tau)-g(s+t)\|^{p} / (1+|s|^{\zeta})^{p} \, ds
\\ \leq & \sum_{k=0}^{\infty}\int^{-k}_{-(k+1)}\|g(s+t+\tau)-g(s+t)\|^{p} / (1+|s|^{\zeta})^{p} \, ds
\\ \leq & 2^{p-1}\|g\|_{S^{p}}^{p}\sum_{k=0}^{\infty}\frac{1}{1+k^{\gamma}}.
\end{align*}
Further on, applying \eqref{medijana} and the H\" older inequality we get that
\begin{align*}
\|u(t+\tau)- u(t)\|
\leq M
\Bigl\|\Bigl(|\cdot|^{\beta-1}(1+|\cdot|^{\zeta})/(1+|\cdot|^{\gamma})\Bigr)\Bigr\|_{L^{q}((-\infty,0))}
\\ 
\times \Biggl[ \int^{0}_{-\infty}\|g(s+t+\tau)-g(s+t)\|^{p} / (1+|s|^{\zeta})^{p} \, ds\Biggr]^{1/p},\ t\in {\mathbb R}.
\end{align*}
Making use of the Fubini theorem and \eqref{austrija}, we get that 
\begin{align*}
\sup_{x\in {\mathbb R}}\Biggl[ & \frac{1}{l}\int^{x+l}_{x}\|u(t+\tau)- u(t)\|^{p}\, dt\Biggr]^{1/p}
\\ \leq & M \Bigl\|\Bigl(|\cdot|^{\beta-1}(1+|\cdot|^{\zeta})/(1+|\cdot|^{\gamma})\Bigr)\Bigr\|_{L^{q}((-\infty,0))}
\\ & \times 
\Biggl[\int^{0}_{-\infty}\frac{1}{(1+|s|^{ \zeta})^{p}}\Biggl(\sup_{x\in {\mathbb R}}\frac{1}{l}\int^{x+l}_{x}\|g(s+t+\tau)-g(s+t)\|^{p}dt \Biggr)  \, ds\Biggr]^{1/p}
\\ \leq & M \epsilon \Bigl\|\Bigl(|\cdot|^{\beta-1}(1+|\cdot|^{\zeta})/(1+|\cdot|^{\gamma})\Bigr)\Bigr\|_{L^{q}((-\infty,0))} \int^{0}_{-\infty}\frac{ds}{(1+|s|^{ \zeta})^{p}},
\end{align*}
for any $l(\epsilon,\tau)>0.$ 
This completes the proof of theorem in a routine manner.
\end{proof}

\begin{rem}\label{jazzine}
Let $1\leq p<\infty.$ Then any equi-Weyl-$p$-almost periodic function is automatically Weyl-$p$-bounded, which seems to be still unknown for Weyl-$p$-almost periodic functions.
\end{rem}

The analysis of asymptotically Stepanov-$p$-almost periodic and asymptotically (equi-)Weyl-$p$-almost periodic properties of finite convolution product is not trivial in general case and we need some extra conditions on the ergodic part of function under consideration (denoted henceforth by $q(\cdot)$) in order to obtain any relevant result in this direction; the situation is, unfortunately, similar if the resolvent family $(R(t))_{t>0}\subseteq L(X,Y)$ satisfies the estimate \eqref{isti-e} or \eqref{isti-ee}. In this paper, we will prove the following general proposition with regards to this question:

\begin{prop}\label{brzo-kusowas}
Let $q\in L^{p}_{loc}([0,\infty) : X),$ $1/p+1/q=1$ and let $(R(t))_{t>0}\subseteq L(X,Y)$ satisfy \eqref{isti-ee}. Let a function $g : {\mathbb R} \rightarrow X$ be (equi-)Weyl-$p$-almost periodic and Weyl $p$-bounded, and let $q(\beta-1)>-1$ provided that $p>1$, resp. $\beta=1,$ provided that $p=1$. 
Suppose that the function
$$
t\mapsto  Q(t)\equiv \int^{t}_{0}R(t-s)q(s)\, ds,\ t\geq 0
$$
belongs to the space ${\mathcal F}_{Y},$ which equals to $C_{0}([0,\infty):Y),$ $S^{p}_{0}([0,\infty):Y),$ $e-W^{p}_{0}([0,\infty):Y)$ or $W^{p}_{0}([0,\infty):Y).$
Then the function
$$
H(t)\equiv \int^{t}_{0}R(t-s)[g(s)+q(s)]\, ds,\ t\geq 0
$$
is continuous and 
belongs to the class $(e-)W_{ap}^{[0,\infty),p}(Y) +{\mathcal F}_{Y}$, where $(e-)W_{ap}^{[0,\infty),p}(Y)$ stands for the space of all restictions of $Y$-valued (equi-)Weyl-$p$-almost periodic functions from the real line to the interval $[0,\infty).$
\end{prop}

\begin{proof}
Without loss of generality, we may assume that $X=Y.$
Define
\begin{align}\label{mi-smo}
F(t):=\int^{t}_{0}R(t-s)q(s)\, ds-\int^{\infty}_{t}R(s)g(t-s)\, ds,\ t\geq 0.
\end{align}
The local integrability of convolution products in \eqref{mi-smo} follows from the arguments given in the proofs of \cite[Proposition 1.3.4, Proposition 1.3.5]{a43}.
By Theorem \ref{brzo-kuso}, the function
$G : {\mathbb R} \rightarrow X,$ defined by \eqref{profa}, is bounded continuous and (equi-)Weyl-$p$-almost periodic. Due to the facts that
$H(t)=G(t)+F(t),$ $ t\geq 0,$ 
${\mathcal F}_{X}+C_{0}([0,\infty) : X)={\mathcal F}_{X}$ and our assumption that the function
$
Q(\cdot)
$
belongs to the space ${\mathcal F}_{X},$ it suffices to show that the mapping $Q(t)$ is continuous for $t\geq 0$ as well as that
the mapping $t\mapsto \int^{\infty}_{t}R(s)g(t-s)\, ds,$ $t\geq 0$ is in class $C_{0}([0,\infty) : X).$ The continuity of mapping $Q(t)$ for $t\geq 0$ can be proved as in the final part of proof of Theorem \ref{brzo-kuso}, by using the equality 
$Q(t)=\int^{t}_{0}R(s)q(t-s)\, ds,$ $t\geq 0,$ inclusion $q\in L^{p}_{loc}([0,\infty) : X)$
and
H\" older inequality.
To prove that the mapping $t\mapsto \int^{\infty}_{t}R(s)g(t-s)\, ds,$ $t\geq 0$ is in class $C_{0}([0,\infty) : X)$, observe first that 
$$
\int^{\infty}_{t}R(s)g(t-s)\, ds=\int^{\infty}_{t+1}R(s)g(t-s)\, ds
+\int^{t+1}_{t}R(s)g(t-s)\, ds, \ t\geq 0.
$$ 
The continuity of mapping $t\mapsto \int^{\infty}_{t+1}R(s)g(t-s)\, ds=\sum_{k=0}^{\infty}\int^{t+k+2}_{t+k+1}R(s)g(t-s)\, ds:=\sum_{k=0}^{\infty}F_{k}(t),$ $t\geq 0$ can be shown following the lines of the proof of \cite[Proposition 5]{element} since the mapping $F_{k}(\cdot)$
is continuous by the dominated convergence theorem and the series $\sum_{k=0}^{\infty}F_{k}(t)$ converges uniformly in $t\geq 0$ due to the Weierstrass criterion. To prove the continuity of mapping $\int^{\cdot+1}_{\cdot}R(s)g(\cdot -s)\, ds,$ fix a number $t\geq 0$ 
and  a sequence $(t_{n})_{n\in {\mathbb N}}$ in $[t,t+1]$ converging to $t$ as $n\rightarrow +\infty.$ Then an elementary argumentation involving the  H\" older inequality shows that:
\begin{align*}
\Biggl \| & \int^{t_{n}+1}_{t_{n}} R(s)g(t_{n} -s)\, ds -\int^{t+1}_{t}R(s)g(t -s)\, ds \Biggr\| 
\\  \leq & M\Biggl[ \int_{t}^{t_{n}}\frac{s^{\beta -1}}{1+s^{\gamma}}\|g(t-s)\|\, ds +\int_{t+1}^{t_{n}+1}\frac{s^{\beta -1}}{1+s^{\gamma}}\|g(t_{n}-s)\|\, ds
\\ + & \int^{t+1}_{t_{n}}\frac{s^{\beta -1}}{1+s^{\gamma}}\bigl \| g(t_{n}-s)-g(t-s) \bigr\| \, ds\Biggr]
\\ \leq & M \Biggl[ \Biggl\| \frac{\cdot^{\beta-1}}{1+\cdot^{\gamma}}\Biggr\|_{L^{q}[t,t_{n}]} \|g\|_{L^{p}[0,t_{n}-t]}+ \Biggl\|\frac{\cdot^{\beta-1}}{1+\cdot^{\gamma}}\Biggr\|_{L^{q}[t+1,t_{n}+1]} \|g\|_{L^{p}[0,t_{n}-t]} 
\\ + & \Biggl\|\frac{\cdot^{\beta-1}}{1+\cdot^{\gamma}}\Biggr\|_{L^{q}[0,t+2]} \|g(t_{n}-\cdot)-g(t-\cdot)\|_{L^{p}[0,t+2]}\Biggr]
\\ \leq &   M \Biggl[ \Biggl\| \frac{\cdot^{\beta-1}}{1+\cdot^{\gamma}}\Biggr\|_{L^{q}[t,t+1]} \|g\|_{L^{p}[0,t_{n}-t]}+ \Biggl\|\frac{\cdot^{\beta-1}}{1+\cdot^{\gamma}}\Biggr\|_{L^{q}[t+1,t+2]} \|g\|_{L^{p}[0,t_{n}-t]} 
\\ + & \Biggl\|\frac{\cdot^{\beta-1}}{1+\cdot^{\gamma}}\Biggr\|_{L^{q}[0,t+2]} \|g(t_{n}-\cdot)-g(t-\cdot)\|_{L^{p}[0,t+2]}\Biggr].
\end{align*}
The right continuity of mapping $\int^{\cdot+1}_{\cdot}R(s)g(\cdot -s)\, ds$ at point $t$ follows from the evident equalities $\lim_{n\rightarrow +\infty}\|g\|_{L^{p}[0,t_{n}-t]}=0$ and $\lim_{n\rightarrow +\infty}\|g(t_{n}-\cdot)-g(t-\cdot)\|_{L^{p}[0,t+2]}=0,$ while the left continuity can be proved analogously. The vanishing of function $t\mapsto \int^{\infty}_{t}R(s)g(t-s)\, ds,$ $t\geq 0$ at plus infinity follows from the estimates 
$$
\Biggl \| \int^{\infty}_{t}R(s)g(t-s)\, ds \Biggr \| \leq \|g\|_{S^{p}}\sum_{k=0}^{\infty}\|R(\cdot)\|_{L^{q}[t+k,t+k+1]}, \ t>0
$$
and
$$
\lim_{t\rightarrow +\infty}\sum_{k=0}^{\infty}\|R(\cdot)\|_{L^{q}[t+k,t+k+1]}=0;
$$
see \cite[Remark 2.14(ii)]{EJDE} and the proof of \cite[Proposition 2.5]{irkutsk}.
\end{proof}

\begin{rem}\label{spejsi}
The space $(e-)W_{ap}^{[0,\infty)}(Y)$ is contained in $(e-)W_{ap}^{p}([0,\infty) : Y).$ It is not clear whether an (equi-)Weyl-$p$-almost periodic function defined on $[0,\infty)$ can be extended to an (equi-)Weyl-$p$-almost periodic function defined on ${\mathbb R}.$ Therefore, it is not clear whether $(e-)W_{ap}^{p}([0,\infty) : Y)\subseteq (e-)W_{ap}^{[0,\infty)}(Y) 
.$
\end{rem}

\begin{rem}\label{ravinjo}
Suppose that $q\in C_{0}([0,\infty) : X)$ and $(R(t))_{t>0}\subseteq L(X,Y)$ satisfies \eqref{isti-ee}. The argumentation contained in the proof of \cite[Lemma 2.13]{ravi} combined with the fact that
$\int^{\infty}_{0}s^{\beta -1}/(1+s^{\gamma})\, ds<\infty$
shows that
$\lim_{t\rightarrow +\infty}Q(t)=0,$ so that $Q\in C_{0}([0,\infty) : Y).$ Some conditions on the function $q(\cdot)$ which ensure that the function $Q(\cdot)$ is Stepanov $p$-vanishing have been analyzed in \cite[Remark 2.14]{EJDE}.
\end{rem}

\begin{rem}\label{spepanov-i}
In \cite{nsjom}, the second named author has recently introduced various classes of generalized (asymptotically) $C^{(n)}$-almost periodic functions and investigated the invariance of generalized (asymptotical) $C^{(n)}$-almost periodicity under the action of convolution products; in particular, the class of (asymptotically) $C^{(n)}$-Weyl-$p$-almost periodic functions has been  
introduced and analyzed. It is worth noting that Theorem \ref{brzo-kuso} and Proposition \ref{brzo-kusowas} can be reconsidered for this class of functions, which can be left to the interested readers as an easy exercise (cf., especially, \cite[Proposition 4.2, Proposition 4.5, Proposition 4.7]{nsjom}). 
\end{rem}

\begin{rem}\label{zeljo}
Suppose that $q\in L^{p}_{loc}([0,\infty) : X),$ $1/p+1/q=1,$ the mapping $a : (0,\infty) \rightarrow (0,\infty)$ satisfies $0<a(t)<t,$ $t>0$ as well as that $(R(t))_{t>0}\subseteq L(X,Y)$ satisfies \eqref{isti-ee}. 
Let $p>1,$ let $B_{p}(0):=0$ and 
\begin{align*}
B_{p}(t):=a(t)^{1/q}&\bigl(t-a(t)\bigr)^{\beta -1-\gamma}\Biggl( \int^{a(t)}_{0}\|q(s)\|^{p}\, ds\Biggr)^{1/p} 
\\ + & \bigl(t-a(t)\bigr)^{\beta -1-\gamma+\frac{1}{q}}\Biggl( \int_{a(t)}^{t}\|q(s)\|^{p}\, ds\Biggr)^{1/p} ,\ t>0.
\end{align*}
Since $Q(t)=\int^{t}_{0}R(t-s)q(s)\, ds=\int^{a(t)}_{0}R(s)q(t-s)\, ds+\int_{a(t)}^{t}R(s)q(t-s)\, ds$ for all $t\geq 0,$ applying the H\" older inequality
and
\eqref{isti-ee} we may conclude that 
$
\|Q(t)\|_{Y}\leq B_{p}(t)$ for all $t>0.
$
In the case that $p=1,$ set $B_{1}(0):=0$ and 
$$
B_{1}(t):=\bigl(t-a(t)\bigr)^{\beta -1-\gamma}\int^{a(t)}_{0}\|q(s)\|\, ds
+\int_{a(t)}^{t}\|q(s)\|\, ds ,\ t>0. 
$$
Then we can similarly prove that $
\|Q(t)\|_{Y}\leq B_{1}(t)$ for all $t>0.
$
This information can be useful to describe the long time behaviour of function $Q(\cdot).$ 

For example, the function
$$
q(t):=\sum_{n=0}^{\infty}\chi_{[n^{2},n^{2}+1]}(t),\quad t\geq 0
$$ 
is not Stepanov $p$-vanishing but it is equi-Weyl-$p$-vanishing for any finite number $p\geq 1$ (using a mollification, we can simply addapt this example to construct an example of an equi-Weyl-$p$-vanishing function that belongs to the space $C^{\infty}([0,\infty))$ and that is not Stepanov $p$-vanishing). Furthermore, $\int^{t}_{0}\|q(s)\|^{p}\, ds \leq 2+\sqrt{t},$ $t\geq 0$ and we can use this estimate as well as the estimate obtained in the first part of this remark (with $a(t)=t/2,$ $t>0$) to see that $Q\in C_{0}([0,\infty) : Y),$ provided that the inequality $\frac{1}{p}+\beta-\gamma <0$ holds true. 

The interested reader may try to construct some examples in which we have that the function $Q(\cdot)$ belongs to the class $e-W^{p}_{0}([0,\infty):Y)$ or $W^{p}_{0}([0,\infty):Y).$ 
\end{rem}

\section[An application]{An application}\label{prcko-duo}

For the beginning, we need to remind ourselves of the notion of a multivalued linear operator (cf. the monographs \cite{cross} by R. Cross and
\cite{faviniyagi} by A. Favini, A. Yagi for more details on the subject). A multivalued map (multimap) ${\mathcal A} : X \rightarrow P(Y)$ is said to be a multivalued
linear operator (MLO) iff the following holds:
\begin{itemize}
\item[(i)] $D({\mathcal A}) := \{x \in X : {\mathcal A}x \neq \emptyset\}$ is a linear subspace of $X$;
\item[(ii)] ${\mathcal A}x +{\mathcal A}y \subseteq {\mathcal A}(x + y),$ $x,\ y \in D({\mathcal A})$
and $\lambda {\mathcal A}x \subseteq {\mathcal A}(\lambda x),$ $\lambda \in {\mathbb C},$ $x \in D({\mathcal A}).$
\end{itemize}
If $X=Y,$ then we say that ${\mathcal A}$ is an MLO in $X.$

If $x,\ y\in D({\mathcal A})$ and $\lambda,\ \eta \in {\mathbb C}$ with $|\lambda| + |\eta| \neq 0,$ then 
$\lambda {\mathcal A}x + \eta {\mathcal A}y = {\mathcal A}(\lambda x + \eta y).$ Furthermore, if ${\mathcal A}$ is an MLO, then ${\mathcal A}0$ is a linear submanifold of $Y$
and ${\mathcal A}x = f + {\mathcal A}0$ for any $x \in D({\mathcal A})$ and $f \in {\mathcal A}x.$ Put $R({\mathcal A}):=\{{\mathcal A}x :  x\in D({\mathcal A})\}.$
Then the set ${\mathcal A}^{-1}0 = \{x \in D({\mathcal A}) : 0 \in {\mathcal A}x\}$ is called the kernel
of ${\mathcal A}$ and it is denoted by $N({\mathcal A}).$ The inverse ${\mathcal A}^{-1}$ of an MLO is defined by
$D({\mathcal A}^{-1}) := R({\mathcal A})$ and ${\mathcal A}^{-1} y := \{x \in D({\mathcal A}) : y \in {\mathcal A}x\}$.

Let ${\mathcal A}$ be an MLO in $X.$
Then
the resolvent set of ${\mathcal A},$ $\rho({\mathcal A})$ for short, is defined as the union of those complex numbers
$\lambda \in {\mathbb C}$ for which
\begin{itemize}
\item[(i)] $R(\lambda-{\mathcal A})=X$;
\item[(ii)] $(\lambda - {\mathcal A})^{-1}$ is a single-valued linear continuous operator on $X.$
\end{itemize}
The operator $\lambda \mapsto (\lambda -{\mathcal A})^{-1}$ is called the resolvent of ${\mathcal A}$ ($\lambda \in \rho({\mathcal A})$); $R(\lambda : {\mathcal A})\equiv  (\lambda -{\mathcal A})^{-1}$  ($\lambda \in \rho({\mathcal A})$). 
The basic properties of resolvents of single-valued linear operators continue to hold in the multivalued linear setting  (\cite{faviniyagi}).

In the remaining part of paper, it will be supposed that ${\mathcal A}$ is a multivalued linear operator on a Banach space $X$ satisfying the following condition (see e.g. \cite[p. 47]{faviniyagi}):
\begin{itemize} 
\item[(P)]
There exist finite constants $c,\ M>0$ and $\beta \in (0,1]$ such that
$$
\Psi:=\Psi_{c}:=\Bigl\{ \lambda \in {\mathbb C} : \Re \lambda \geq -c\bigl( |\Im \lambda| +1 \bigr) \Bigr\} \subseteq \rho({\mathcal A})
$$
and
$$
\| R(\lambda : {\mathcal A})\| \leq M\bigl( 1+|\lambda|\bigr)^{-\beta},\quad \lambda \in \Psi.
$$
\end{itemize}

Suppose that $\gamma \in (0,1)$ and $\beta>\theta .$
Then the degenerate strongly continuous semigroup $(T(t))_{t> 0}\subseteq L(X)$ generated by ${\mathcal A}$ satisfies the estimate
\begin{align}\label{prva-est}
\|T(t) \| \leq M_{0} e^{-ct}t^{\beta -1},\ t> 0 
\end{align}
for some finite constant $M_{0}>0$ (\cite{EJDE}). Define 
\begin{align*}
T_{\gamma,\nu}(t)x:=t^{\gamma \nu}\int^{\infty}_{0}s^{\nu}\Phi_{\gamma}( s)T\bigl( st^{\gamma}\bigr)x\, ds,\ t>0,\ x\in X,
\end{align*}
\begin{align*}
S_{\gamma}(t):=T_{\gamma,0}(t),\ t>0 \ \ ; \ \ S_{\gamma}(0):=I,
\end{align*}
\begin{align*}
P_{\gamma}(t):=\gamma T_{\gamma,1}(t)/t^{\gamma},\ t>0\ \ ; \ \
R_{\gamma}(t):=t^{\gamma -1}P_{\gamma}(t),\ t>0.
\end{align*}
Then there exist two finite constants $M_{1}>0$ and $M_{2}>0$ such that
\begin{align}\label{druga-est}
\bigl\| R_{\gamma}(t) \bigr\|\leq M_{1}t^{\gamma \beta -1},\ t\in (0,1] \ \ \mbox{ and }\ \ \bigl\| R_{\gamma}(t) \bigr\|\leq M_{2}t^{-1-\gamma},\ t\geq 1.
\end{align}

We will use the following notions:
A continuous function $u: {\mathbb R} \rightarrow X$ is a mild solution of the abstract first order inclusion
\begin{align*}
u(t)\in {\mathcal A}u(t)+f(t),\ t\in {\mathbb R}
\end{align*}
iff 
$$
u(t)=\int^{t}_{-\infty}T(t-s)f(s)\, ds,\ t\in {\mathbb R}.
$$

Let $\gamma \in (0,1).$ A continuous function $u: {\mathbb R} \rightarrow X$ is a mild solution of the abstract fractional relaxation inclusion
\begin{align}\label{inkpow}
D_{t,+}^{\gamma}u(t)\in -{\mathcal A}u(t)+f(t),\ t\in {\mathbb R}
\end{align}
iff 
$$
u(t)=\int^{t}_{-\infty}R_{\gamma}(t-s)f(s)\, ds,\ t\in {\mathbb R}.
$$

Let ${\mathbf D}_{t}^{\gamma}$ denote the Caputo fractional derivative of order $\gamma \in (0,1),$ let $x_{0}\in X$ and let 
$f : [0,\infty) \rightarrow X$ be asymptotically (equi-)Weyl-$p$-almost periodic. Suppose, further, that
$x_{0}$ is a point of continuity of $(S_{\gamma}(t))_{t>0},$ i.e., $\lim_{t\rightarrow 0+}S_{\gamma}(t)x_{0}=x_{0}.$
By a mild solution of the abstract fractional relaxation inclusion
\[
\hbox{(DFP)}_{f,\gamma} : \left\{
\begin{array}{l}
{\mathbf D}_{t}^{\gamma}u(t)\in {\mathcal A}u(t)+f(t),\ t> 0,\\
\quad u(0)=x_{0},
\end{array}
\right.
\]
we mean any function $u\in C([0,\infty) : X)$ satisfying that 
\begin{align*}
u(t)=S_{\gamma}(t)x_{0}+\int^{t}_{0}R_{\gamma}(t-s)f(s)\, ds,\quad t\geq 0.
\end{align*}

Keeping in mind the estimates \eqref{prva-est}-\eqref{druga-est} and the fact that $\lim_{t\rightarrow +\infty}\|S_{\gamma}(t)\|=0,$ it is clear how we can apply the main results of Section 3 in the study of existence and uniqueness of (equi-)Weyl-$p$-almost periodic solutions of the abstract fractional inclusion \eqref{inkpow} and 
asymptotically (equi-)Weyl-$p$-almost periodic solutions of the abstract fractional inclusion (DFP)$_{f,\gamma}$ (solutions of problem (DFP)$_{f,1},$ with ${\mathbf D}_{t}^{1}u(t)\equiv u^{\prime}(t)$ and the meaning clear, can be also examined). Applications can be simply incorporated in the study of qualitative properties of solutions of the following fractional Poisson heat equations with the Dirichlet Laplacian $\Delta$:
\[\left\{
\begin{array}{l}
D_{t,+}^{\gamma}[m(x)v(t,x)]=(\Delta -b )v(t,x) +f(t,x),\ t\in {\mathbb R},\ x\in {\Omega};\\
v(t,x)=0,\quad (t,x)\in [0,\infty) \times \partial \Omega ,\\
\end{array}
\right.
\]
and
\[\left\{
\begin{array}{l}
{\mathbf D}_{t}^{\gamma}[m(x)v(t,x)]=(\Delta -b )v(t,x) +f(t,x),\ t\geq 0,\ x\in {\Omega};\\
v(t,x)=0,\quad (t,x)\in [0,\infty) \times \partial \Omega ,\\
 m(x)v(0,x)=u_{0}(x),\quad x\in {\Omega},
\end{array}
\right.
\]
in the space $X:=L^{p}(\Omega),$ where $\Omega$ is a bounded domain in ${\mathbb R}^{n}$ with smooth boundary, $b>0,$ $m(x)\geq 0$ a.e. $x\in \Omega$, $m\in L^{\infty}(\Omega),$ $\gamma \in (0,1)$ and $1<p<\infty .$ Further on, let
$A(x;D)$ be a second order linear differential operator on $\Omega$ with coefficients continuous on $\overline{\Omega};$
see \cite[Example 6.1]{faviniyagi} for more details. Based on the examination carried out in \cite[Example 4]{element}, we can apply our main results in the study of existence and uniqueness of
asymptotically (equi-)Weyl-$p$-almost periodic solutions of the following fractional damped Poisson-wave type equation\index{equation!damped Poisson-wave} in the space $X:=H^{-1}(\Omega)$ or $X:=L^{p}(\Omega):$
\[\left\{
\begin{array}{l}
{\mathbf D}_{t}^{\gamma}\bigl( m(x){\mathbf D}_{t}^{\gamma}u\bigr)+\bigl(2\omega m(x)-\Delta \bigr){\mathbf D}_{t}^{\gamma}u+\bigl(A(x;D)-\omega \Delta+\omega^{2}m(x)\bigr)u(x,t)=f(x,t),\\ t\geq 0,\ x\in \Omega \ \ ; \ \
u={\mathbf D}_{t}^{\gamma}=0,\quad (x,t)\in \partial \Omega \times [0,\infty),\\
u(0,x)=u_{0}(x),\ m(x)\bigl[ {\mathbf D}_{t}^{\gamma}u(x,0)+\omega u_{0}\bigr]=m(x)u_{1}(x),\quad x\in {\Omega}.
\end{array}
\right.
\]

\bibliographystyle{amsplain}

\end{document}